\documentclass{amsart}

\usepackage{amssymb}
\usepackage{amsmath}
\usepackage{amsthm}

\newtheorem{theorem}{Theorem}[section]
\newtheorem{lemma}[theorem]{Lemma}
\newtheorem{corollary}[theorem]{Corollary}
\newtheorem{proposition}[theorem]{Proposition}

\theoremstyle{remark}

\begin{document}
\title[$K(s)^*(BG)$, $G=(C_{2^{n+1}}\times C_{2^{n+1}})\rtimes C_2$]{All extensions of $C_2$
by $C_{2^{n+1}}\times C_{2^{n+1}}$ are good}
\author{Malkhaz Bakuradze}
\address{Iv. Javakhishvili Tbilisi State University, Faculty of Exact and Natural Sciences}
%\curraddr{}
\email{malkhaz.bakuradze@tsu.ge}
%\thanks{The author was supported by Volkswagen Foundation, Ref. 1/84 328 and Rustaveli Foundation grant DI/16/5-103/12.}

%\dedicatory{}

%    "Communicated by" -- provide editor's name; required.
%\commby{}
\date{}

%    Abstract is required.

\begin{abstract}
Let $C_{m} $ be a cyclic group of order $m$.  We prove that if the group $G$ fits into an extension $1\to C_{2^{n+1}}^2\to G\to C_2\to 1$ then $G$ is good in the sense of Hopkins-Kuhn-Ravenel, i.e., $K(s)^*(BG)$
is evenly generated by transfers of Euler classes of complex representations of subgroups of $G$. Previously this fact was known for $n=1$. 
\end{abstract}

\maketitle{}

\section{Introduction and Statements}
This paper is concerned with analyzing the 2-primary Morava $K$-theory of the classifying spaces $BG$ of the groups in the title. In particular it answers  the question whether transfers of Euler classes suffice to generate $K(s)^*(BG)$. Here $K(s)$ denotes  Morava K-theory at prime $p=2$ and natural $s>1$.  The coefficient ring $K(s)^*(pt)$ is the Laurent polynomial ring in one variable, $\mathbb{F}_2[v_s,v_s^{-1}]$, where $\mathbb{F}_2$ is the field of 2 elements and $deg(v_s)=-2(2^s-1)$ \cite{JW}. So the coefficient ring is a graded field in the sense that all its graded modules are free, therefore Morava $K$-theories enjoy the K\"{u}nneth isomorphism. In particular, we have for the cyclic group $C_{2^{n+1}}$ that as a $K(s)^*$-algebra
$$K(s)^*(BC_{2^{n+1}}^2)=K(s)^*(BC_{2^{n+1}})\otimes K(s)^*(BC_{2^{n+1}}),$$
whereas $K(s)^*(BC_{2^m})=K(s)^*[u]/(u^{2^{ms}})$. So that
$$K(s)^*(BC_{2^{n+1}}^2)=K(s)^*[u,v]/(u^{2^{(n+1)s}},v^{2^{(n+1)s}}),$$
where $u$ and $v$ are Euler classes of canonical complex linear representations.

\medskip

The definition of good groups in the sense \cite{HKR} is as follows.

a) For a finite group $G$, an element $x\in K(s)^*(BG)$ is good
if it is a transferred Euler class of a complex subrepresentation of $G$, i.e., a class
of the form $Tr^*(e(\rho))$, where $\rho$ is a complex representation of a subgroup $H<G$, $e(\rho)\in K(s)^*(BH)$ is its Euler class (i.e., its top Chern class, this being defined since $K(s)^*$ is a complex oriented theory), and $Tr:BG \to BH$ is the transfer map.

(b) $G$ is called to be good if $K(s)^*(BG)$ is spanned by good elements as a $K(s)^*$–module.

Recall not all finite groups are good as it was originally conjectured in \cite{HKR}. For odd prime $p$ a counterexample to the even degree was constructed In \cite{K}.   
The problem to construct $2$-primary counterexample  conjecture remains open.

Our main result is as follows

\medskip

\begin{theorem}
\label{main} 
All extensions of $C_2$ by  $C_{2^{n+1}}^2$ are good.
\end{theorem}

%In other words if $G$ fits into an extension 
%$
%1\to C_{2^{n+1}}^2\to G\to C_2\to 1, \,\,\,n\geq 2,
%$
%then $G$ is good. 

For $n=1$ the statement of the theorem was known. 
See  \cite{B-GMJ2}, \cite{BJ}, \cite{SCH1}, \cite{SCHY} for detailed discussion and examples.

Of course the Serre spectral sequence is used throughout the paper. However, if to operate straightforward, even  for $s=2$, $n=1$, this requires a serious computational effort and use of computer, see \cite{SCH4} p.78. We simplify the task of calculation with invariants by suggesting the special bases for particular $C_2$-modules $K(s)^*(BH)$,  see Lemma \ref{specialbase} and Lemma \ref{invariants}.
This simple but comfortable idea is our key tool to prove Theorem \ref{main}. We will prove it for the semi-direct products 
\begin{equation}
\label{semi}
(C_{2^{n+1}}\times C_{2^{n+1}})\rtimes C_2.
\end{equation} 
Then the general case follows because of the fact that the Serre spectral sequence does not see the difference between the semi-direct products and their non-split versions.

%We also mention \cite{Y3}.  

\section{Preliminaries}

Recall \cite{G-P} there exist exactly 17 non-isomorphic groups of order $2^{2n+3}$, $n\geq 2$, which can be presented as a semidirect product
\eqref{semi}. Each such group G is given by three generators $\mathbf{a}, \mathbf{b}, \mathbf{c}$ and the defining relations
$$\mathbf{a}^{2^{n+1}}=\mathbf{b}^{2^{n+1}}=\mathbf{c}^2 = 1, \mathbf{a}\mathbf{b} = \mathbf{b}\mathbf{a}, \mathbf{c}^{-1}\mathbf{a}\mathbf{c} =\mathbf{a}^i\mathbf{b}^j, \mathbf{c}^{-1}\mathbf{b}\mathbf{c} =\mathbf{a}^k\mathbf{b}^l$$
for some $i, j, k, l \in Z_{2^{n+1}}$ ($Z_{2^m}$ denotes the ring of residue classes modulo $2^m$). In particular one has the following
\begin{proposition} (see \cite{G-P} )
	\label{Euler}
	Let $n$ be an integer such that $n \geq 2$. Then there exist exactly 17 non-isomorphic groups of order $2^{2n+3}$
	which can be presented as a semi-direct product \eqref{semi}. They are:
	
	\begin{align*}
	&G_1=\langle \mathbf{a},\mathbf{b},\mathbf{c} \mid (*), \mathbf{c}\mathbf{a}\mathbf{c}=\mathbf{a}, \mathbf{c}\mathbf{b}\mathbf{c}=\mathbf{b}
	\rangle  , \\
	&G_2=\langle \mathbf{a},\mathbf{b},\mathbf{c} \mid (*), \mathbf{c}\mathbf{a}\mathbf{c}=\mathbf{a}^{1+2^{n}}, \mathbf{c}\mathbf{b}\mathbf{c}=\mathbf{b}^{1+2^{n}} \rangle ,\\
	&G_3=\langle \mathbf{a},\mathbf{b},\mathbf{c} \mid (*), \mathbf{c}\mathbf{a}\mathbf{c}=\mathbf{a}\mathbf{b}^{2^{n}}, \mathbf{c}\mathbf{b}\mathbf{c}=\mathbf{b} \rangle ,\\
	&G_4=\langle \mathbf{a},\mathbf{b},\mathbf{c} \mid (*), \mathbf{c}\mathbf{a}\mathbf{c}=\mathbf{a}^{1+2^{n}}\mathbf{b}^{2^{n}}, \mathbf{c}\mathbf{b}\mathbf{c}=\mathbf{b}^{1+2^{n}} \rangle ,\\
	&G_5=\langle \mathbf{a},\mathbf{b},\mathbf{c} \mid (*), \mathbf{c}\mathbf{a}\mathbf{c}=\mathbf{a}^{-1}, \mathbf{c}\mathbf{b}\mathbf{c}=\mathbf{b}^{-1} \rangle ,\\
	&G_6=\langle \mathbf{a},\mathbf{b},\mathbf{c} \mid (*), \mathbf{c}\mathbf{a}\mathbf{c}=\mathbf{a}^{-1+2^{n}}, \mathbf{c}\mathbf{b}\mathbf{c}=\mathbf{b}^{-1+2^{n}} \rangle ,\\
	&G_7=\langle \mathbf{a},\mathbf{b},\mathbf{c} \mid (*), \mathbf{c}\mathbf{a}\mathbf{c}=\mathbf{a}^{-1}\mathbf{b}^{2^{n}}, \mathbf{c}\mathbf{b}\mathbf{c}=\mathbf{b}^{-1} \rangle ,\\
	&G_8=\langle \mathbf{a},\mathbf{b},\mathbf{c} \mid (*), \mathbf{c}\mathbf{a}\mathbf{c}=\mathbf{a}^{-1+2^{n}}\mathbf{b}^{2^{n}}, \mathbf{c}\mathbf{b}\mathbf{c}=\mathbf{b}^{-1+2^{n}} \rangle ,\\
	&G_9=\langle \mathbf{a},\mathbf{b},\mathbf{c} \mid (*), \mathbf{c}\mathbf{a}\mathbf{c}=\mathbf{a}\mathbf{b}^{2^{n}}, \mathbf{c}\mathbf{b}\mathbf{c}=\mathbf{a}^{2^{n}}\mathbf{b}^{1+2^{n}} \rangle ,\\
	&G_{10}=\langle \mathbf{a},\mathbf{b},\mathbf{c} \mid (*), \mathbf{c}\mathbf{a}\mathbf{c}=\mathbf{a}, \mathbf{c}\mathbf{b}\mathbf{c}=\mathbf{b}^{1+2^{n}} \rangle ,\\
	&G_{11}=\langle \mathbf{a},\mathbf{b},\mathbf{c} \mid (*), \mathbf{c}\mathbf{a}\mathbf{c}=\mathbf{a}^{-1}\mathbf{b}^{2^{n}}, \mathbf{c}\mathbf{b}\mathbf{c}=\mathbf{a}^{2^{n}}\mathbf{b}^{-1+2^{n}} \rangle ,\\
	&G_{12}=\langle \mathbf{a},\mathbf{b},\mathbf{c} \mid (*), \mathbf{c}\mathbf{a}\mathbf{c}=\mathbf{a}^{-1}, \mathbf{c}\mathbf{b}\mathbf{c}=\mathbf{b}^{-1+2^{n}} \rangle ,\\
	&G_{13}=\langle \mathbf{a},\mathbf{b},\mathbf{c} \mid (*), \mathbf{c}\mathbf{a}\mathbf{c}=\mathbf{a}, \mathbf{c}\mathbf{b}\mathbf{c}=\mathbf{b}^{-1+2^{n}} \rangle ,\\
	&G_{14}=\langle \mathbf{a},\mathbf{b},\mathbf{c} \mid (*), \mathbf{c}\mathbf{a}\mathbf{c}=\mathbf{a}^{-1}, \mathbf{c}\mathbf{b}\mathbf{c}=\mathbf{b}^{1+2^{n}} \rangle ,\\
	&G_{15}=\langle \mathbf{a},\mathbf{b},\mathbf{c} \mid (*), \mathbf{c}\mathbf{a}\mathbf{c}=\mathbf{b}, \mathbf{c}\mathbf{b}\mathbf{c}=\mathbf{a} \rangle ,\\
	&G_{16}=\langle \mathbf{a},\mathbf{b},\mathbf{c} \mid (*), \mathbf{c}\mathbf{a}\mathbf{c}=\mathbf{a}, \mathbf{c}\mathbf{b}\mathbf{c}=\mathbf{b}^{-1} \rangle ,\\
	&G_{17}=\langle \mathbf{a},\mathbf{b},\mathbf{c} \mid (*), \mathbf{c}\mathbf{a}\mathbf{c}=\mathbf{a}^{1+2^{n}}, \mathbf{c}\mathbf{b}\mathbf{c}=\mathbf{b}^{-1+2^{n}} \rangle ,\\
	\end{align*}
	where (*) denotes the collection \{$\mathbf{a}^{2^{n+1}}=\mathbf{b}^{2^{n+1}}=\mathbf{c}^2=[\mathbf{a},\mathbf{b}]=1$ \} of defining relations.
\end{proposition}

\bigskip

In \cite{B1} we proved \begin{theorem}
	\label{1}
	Let $H_i$ and $G_i$ be finite $p$-groups, $i=1,\dots, n$, such that $H_i$ is good and  $G_i$ fits into an extension $1\to H_i\to G_i\to C_p \to 1$.
	
	Let $G$ fit into an extension of the form $1\to H \to G\to C_p \to 1$, with diagonal action of $C_p$ by conjugation on $H=H_1 \times\cdots\times H_n $.  Denote by
	$$
	Tr^*=Tr^*_{\varrho}: K(s)^*(BH)\to K(s)^*(BG),
	$$
	the transfer homomorphism associated to the $p$-covering $\varrho=\varrho(H,G):BH \to BG$,
	$$
	Tr_i^*=Tr^*_{\varrho_i}:K(s)^*(BH_i)\to K(s)^*(BG_i),
	$$
	the transfer homomorphism associated to the $p$-covering  $\varrho_i=\varrho(H_i,G_i):BH_i \to BG_i$, $i=1,\dots ,n$,
	$$
	\rho_i:BG\to BG_i ,
	$$
	the map, induced by the projection $H \to H_i$ on the $i$-th factor,
	and let $\rho^*$ be the restriction of
	$$(\rho_1,\cdots ,\rho_n)^*:K(s)^*(BG_1\times \cdots \times BG_n)\to K(s)^*(BG)$$
	on $K(s)^*(BG_1)/Im Tr^*_1 \otimes \cdots \otimes K(s)^*(BG_n)/ImTr^*_n$.  Then
	
	\begin{enumerate}
		\item[\rm i)] If $G_i$ are good and so is $G$.
		
		\item[\rm ii)] $K(s)^*(BG)$ is spanned, as a $K(s)^*(pt)$-module,
		by the elements of $ImTr^*$ and $Im\rho^*$.
	\end{enumerate}
\end{theorem}

In particular this implies

\begin{corollary}
	\label{easy-case}
	Let $G=G_i$, $i\neq 3,4,7,8,9,11$,
	%$i=1,2,5,6,10,12,13,14,15,16,17$,
	then $G$ is good in the sense of Hopkins-Kuhn-Ravenel.
\end{corollary}
\begin{proof} $G_{15}$ is good as wreath product \cite{HKR}. Otherwise   
	  $G_i$ has maximal abelian subgroup $H_i=\langle \mathbf{a},\mathbf{b}\rangle$ on which the quotient acts (diagonally) as above. Each of the following groups $C_{2^{n+1}}\times C_2$, the dihedral group $D_{2^{n+2}}$, the quasi-dihedral group $QD_{2^{n+2}}$, the semi-dihedral group $SD_{2^{n+2}}$  could be written as semidirect product $C_{2^{n+1}}\rtimes C_2$ with that kind of action.  For all these  groups $K(s)^*(BG)$ is generated by Euler classes, see \cite{TY1, TY2}.  
%Theorem (Kriz) \cite{K} tell us that that %$M_j:=\tilde{K}(s)^*(BC_{2^{n+1}})$ is a permutation module for each such %automorphism, thus
%	$\tilde{K}(s)^*(BH_i)\cong M_j\otimes M_k$ is a permutation module too. %Theorem (Kriz)again thus implies $K(s)^{odd}(BG_i)=0$. 
\end{proof}

\bigskip

\medskip

\medskip

We will need the following approximations (see \cite{BV}, Lemma 2.2) for the formal group law in Morava $K(s)^*$-theory, $s>1$, we set $v_s=1$. 

\begin{equation}
\label{FGL2^s}
F(x,y)=x+y+(xy)^{2^{s-1}}, \,\,\,\mod{(y^{2^{2(s-1)}})}; 
\end{equation}

\begin{equation}
\label{FGL}
F(x,y)=x+y+\Phi(x,y)^{2^{s-1}},
\end{equation}
where $\Phi(x,y)=xy+(xy)^{2^{s-1}}(x+y)\,\,\,  \mod((xy)^{2^{s-1}}(x+y)^{2^{s-1}}).$

\bigskip

\section{Complex representations over $BG$}

Let us define some complex representations over $BG$ we will need.

\medskip

Let $H=\langle \mathbf{a},\mathbf{b}\rangle\cong C_{2^{n+1}} \times C_{2^{n+1}} $ be the maximal abelian subgroup in $G$.
Let 
\begin{equation}
\label{pi}
\pi:BH \rightarrow BG
\end{equation}
be the double covering.
Let $\lambda$ and $\nu$ denote the following complex line bundles over $BH$ 
$$\lambda(\mathbf{a})=\nu(\mathbf{b})=e^{2\pi i/2^{n+1} },\,\,\, \lambda(\mathbf{b})=\lambda(\mathbf{c})=\nu(\mathbf{a})=\nu(\mathbf{c})=1,$$
be the pullbacks of the canonical complex line bundles along the projections onto the first and second  factor of $H$ respectively.
Let 
$$\pi_!(\lambda)=Ind^G_H(\lambda) \text{ and } \,\, \pi_!(\nu)=Ind^G_H(\nu)$$ 
be the plane bundles over $BG$, the transferred $\lambda$ and $\nu$ respectively. Then define three line bundles over $BG$,
$\alpha$, $\beta$ and $\gamma$ as follows

$$\alpha(\mathbf{a})=\beta(\mathbf{b})=\gamma(\mathbf{c})=-1,\,\,\,\alpha(\mathbf{b})=\alpha(\mathbf{c})=\beta(\mathbf{a})=
\beta(\mathbf{c})=\gamma(\mathbf{a})=\gamma(\mathbf{b})=1.$$

Let us denote Chern classes by

\begin{align*}
&x_i=c_i(Ind_H^G(\lambda)), \,\,\,y_i=c_i(Ind_H^G)(\nu)),\,\,\,i=1,2,\\
&a=c_1(\alpha),\,\,\,b=c_1(\beta),\,\,\,c=c_1(\gamma).\\
\end{align*}

\medskip

\bigskip

\section{Proof of Theorem \ref{main}}

Here we prove that all the remaining groups $G_i$, $i=3,4,7,8,9,11$,
%$i=1,2,5,6,10,12,13,14,15,16,17$, 
not covered by Corollary \ref{easy-case},
are also good.

Our tool shall be the Serre spectral sequence
\begin{equation}
E_2=H^*(BQ),K(s)^*(BH))\Rightarrow K(s)^*(BG)
\end{equation}
associated to a group extension $1\to H \xrightarrow[]{}G \xrightarrow[]{} C_2\to 1$. Here $H^*(BC_2),K(s)^*(BH))$ denotes the ordinary cohomology of $BC_2$ with coefficients in the $\mathbb{F}_2[C_2]$-module $K(s)^*(BH)$, where the action of $C_2$
is induced by conjugation in $G$.

Let $Tr^*: K(s)^*(BH)\to K(s)^*(BG)$ be the transfer homomorphism \cite{A},\cite{KP}, \cite{D}. associated to the double covering $\pi: BH \to BG$. 

\bigskip

 We use the notations of previous two sections. In particular let
	 $$H\cong C_{2^{n+1}}\times C_{2^{n+1}}\cong \langle \mathbf{a},\mathbf{b}\rangle.$$ 

Consider the decomposition
\begin{equation}
\label{free-trivial}
[K(s)^*(BH)]^{C_2}=[F]^{C_2}+T,
\end{equation}
 corresponding to the decomposition of  $K(s)^*(BH)$ into free and trivial $C_2$-modules. The action of the involution $t\in C_2$ on
\begin{equation}
\label{H}
K(s)^*(BH)=K(s)^*[u,v]/(u^{2^{(n+1)s}},v^{2^{(n+1)s}})
\end{equation}
is induced by conjugation action by $\mathbf{c}$ on $H$. 
Clearly the composition $\pi^*Tr^*=1+t$, the trace map, is onto $[F]^{C_2}$. Therefore it suffices to check that all  invariants in $T$ are also represented by good elements. 
	
For all cases of $G$ let  
\begin{align*}
u=&e(\lambda),\,\,\, && v=e(\nu),\text{ as before},\\
\bar{x}_1=&u+t(u)=\pi^*(x_1),
&&\bar{x}_2=ut(u)=\pi^*(x_2),\\
\bar{y}_1=&v+t(v)=\pi^*(y_1),
&&\bar{y}_2=vt(v)=\pi^*(y_2).\\
\end{align*}

%Note that for any action of the involution $t$ one has  
%\begin{equation}
%\label{det}
%F(u,t(u))\equiv(\bar{x}_2)^{2^{s-1}} \mod(\bar{x}_1f(\bar{x}_1,\bar{x}_2)\in %Im(1+t).
%\end{equation}

%This is because the formal group law has a form
%$$F(x,y)=(xy)^{2^{s-1}}+(x+y)(1+S(x,y)), \text{where S is a symmetric %series}.$$

\medskip

We will need the following

\begin{lemma}
\label{specialbase}
Let $G$ be one of the groups under consideration and $t\in C_2= G/H$ be corresponding involution on $H$. Then there is a set of monomials
$\{x^{\omega}\}=\{\bar{x}_1^i\bar{x}_2^j\bar{y}_1^k\bar{y}_2^l\}$, such that 
the set $\{x^{\omega},x^{\omega} u,x^{\omega} v, x^{\omega} uv\}$ is a $K(s)^*$-basis
in $K(s)^*(BH)$.  In particular one can choose $\{x^{\omega}\}$ as follows

\begin{equation*}
\{x^{\omega}\}=
\begin{cases}
\{\bar{x}_2^j\bar{y}_1^k\bar{y}_2^{l}| j<2^{ns-1}, k<2^s,l<2^{(n+1)s-1}\},& \text{ if }\,\,\, G= G_3,\\ 
\{\bar{x}_1^i\bar{x}_2^j\bar{y}_1^k\bar{y}_2^l| i,k<2^s,j,l<2^{ns-1} \}, & \text{ if }\,\,\, G= G_4, G_9,\\
\{\bar{x}_1^i\bar{x}_2^j\bar{y}_1^k\bar{y}_2^l| i,k<2^{ns},j,l<2^{s-1} \}, & \text{ if  }\,\,\, G= G_7, G_8, G_{11}.\\
\end{cases}
\end{equation*}

\end{lemma}

\begin{proof} If ignore the restrictions, the set 	$\{x^{\omega},x^{\omega} u, x^{\omega} v, x^{\omega} uv\}$, generates $K(s)^*(BH)$: using $u^2=u\bar{x_1}-\bar{x_2}$ and $v^2=v\bar{y_1}-\bar{y_2}$ any polynomial in $u,v$ can be written as  $g_0+g_1u+g_2v+g_3uv$ where $g_i= 
g_i(\bar{x_1},\bar{y_1},\bar{x_2},\bar{y_2})$ are some polynomials. In particular it follows by induction, that
\begin{equation}
\label{v^2^m}
u^{2^{m}}=u\bar{y_1}^{2^{m}-1}+
\sum_{i=1}^{m}\bar{y_1}^{2^{m}-2^i}\bar{y_2}^{2^{i-1}},
\end{equation}

and similarly for $v^{2^m}$. 

Now for each case we have to explain the restrictions in $\{x^{\omega}\} $.
Then the restricted set $S=\{x^{\omega}, x^{\omega} u, x^{\omega} v, x^{\omega} uv \}$ will  indeed form a $K^*(s)$-basis in $K^*(s)(BH)$ because of its size $4^{(n+1)s}$.  

Consider $G_3$. For the restrictions on $l$ and $k$ we have to take into account \eqref{FGL2^s}, \eqref{H} and the action of the involution $t$.  
In particular 

%t(\lambda)(a)=\lambda(cac)=\lambda(ab^{2^n})=\lambda(a); %t(\lambda)(b)=\lambda(cbc)=\lambda(b)
%t(\nu)(a)=\nu(cac)=\nu(ab^{2^n})=\nu^{2^n}(b)=-1=(\lambda^{2^n}\nu)(a);  %t(nu)(b)=\nu(cbc)=\nu(b)=(lambda^{2^n}\nu)(b);

\begin{align*}
&t(\lambda)=\lambda ,\,\,\,\, t(\nu)=\lambda^{2^n}\nu,\\
&t(u)=u,\\
&\Rightarrow \bar{x}_1=u+t(u)=0,\\
&\bar{x}_2=ut(u)=u^2,\\
&t(v)=F(u^{2^{ns}},v)=v+u^{2^{ns}}+(vu^{2^{ns}})^{2^{s-1}},\\
&\Rightarrow \bar{y}_2^{2^{(n+1)s-1}}=0\\
&\bar{y}_1=v+t(v)=u^{2^{ns}}+(vu^{2^{ns}})^{2^{s-1}},\\
&\Rightarrow \bar{y}_1^{2^{s}}= 0.\\
\end{align*}

For the restriction on $j$, that is, the decomposition of $\bar{x}_2^{2^{ns-1}}$ in the suggested basis, note that the formula for $t(v)$ and \eqref{v^2^m} for $m=s-1$ imply

 \begin{align*}
 \bar{x}_2^{2^{ns-1}}=&u^{2^{ns}}=\bar{y}_1+(vu^{2^{ns}})^{2^{s-1}}=\\
 &\bar{y}_1+v^{2^{s-1}}(\bar{y}_1+(vu^{2^{ns}})^{2^{s-1}})^{2^{s-1}}\\
 &\bar{y}_1+v^{2^{s-1}}\bar{y}_1^{2^{s-1}}=\\
 &\bar{y}_1+\bar{y}_1^{2^{s-1}}(v\bar{y_1}^{2^{s-1}-1}+
 \sum_{i=1}^{s-1}\bar{y_1}^{2^{s-1}-2^i}\bar{y_2}^{2^{i-1}})=\\
 &\bar{y}_1+v\bar{y}_1^{2^{s}-1}+
 \bar{y}_1^{2^{s-1}}\sum_{i=1}^{s-1}\bar{y_1}^{2^{s-1}-2^i}\bar{y_2}^{2^{i-1}}.\\
 \end{align*}

 \medskip
 
$ G_4 $: The involution acts as follows:
$t(\lambda)=\lambda^{2^n+1},\,\,\,t(\nu) =\lambda^{2^n}\nu^{2^{n}+1}$, hence
%t(\lambda)(a)=\lambda(cac)=\lambda(a^{1+2^n}b^{2^n})=\lambda^{1+2^n}(a);
%t(\lambda)(b)=\lambda(cbc)=\lambda(b^{1+2^n})=1=\lambda^{1+2^n}(b);	
%t(\nu)(a)=\nu(cac)=\nu(a^{1+2^n}b^{2^n})=-1=(\lambda^{2^n}\nu^{1+2^n})(a)
%t(\nu)(b)=\nu(cbc)=\nu(b^{1+2^n})(b)=(\lambda^{2^n}\nu^{1+2^n})(b)	

\begin{align}
\label{xyG4-1}
&t(u)=F(u,u^{2^{ns}})=u+u^{2^{ns}}+(uu^{2^{ns}})^{2^{s-1}} \,\,\,\text{ by }\,\,\, \eqref{FGL2^s},\\
\label{xyG4-2}	&t(v)=F(v,F(v^{2^{ns}},u^{2^{ns}}))=v+F(v^{2^{ns}},u^{2^{ns}})+	v^{2^{s-1}}F(v^{2^{ns}},u^{2^{ns}}))^{2^{s-1}}.
\end{align}

So that $\bar{x}_1^{2^s}=\bar{y}_1^{2^s}=0$.  

For the decomposition of $\bar{x}_2^{2^{ns-1}}$, note \eqref{xyG4-1} implies 
$$\bar{x}_2^{2^{ns-1}}=(ut(u))^{2^{ns-1}}=u^{2^{ns}}.$$
Then by \eqref{xyG4-1} again
$$
\bar{x}_2^{2^{ns-1}}=\bar{x}_1+(u\bar{x}_2^{2^{ns-1}})^{2^{s-1}}=
\bar{x}_1+(u(\bar{x}_1+(u\bar{x}_2^{2^{ns-1}})^{2^{s-1}}))^{2^{s-1}}=
\bar{x}_1+u^{2^{s-1}}\bar{x}_1^{2^{s-1}}
$$
and apply \eqref{v^2^m} for $u^{2^{s-1}}$. 

Similarly for $\bar{y}_2^{2^{ns-1}}$.

\medskip

The proof for $G_9$ is completely analogous as it uses the similar formulas for the action of the involution

%t(\lambda)(a)=\lambda(cac)=\lambda(ab^{2^n})=\lambda(a)=\lambda\nu^{2^n}(a)
%t(\lambda)(b)\lambda(cbc)=\lambda(a^{2^n}b^{1+2^n})=-1=\lambda\nu^{2^n}(b)
%
%t(\nu)(a)=\nu(cac)=\nu(ab^{2^n})=-1=\lambda^{2^n}\nu^{1+2^n}(a)
%t(\nu)(b)=\nu(cbc)=\nu(a^{2^n}b^{1+2^n})=-\nu(b)=\lambda^{2^n}\nu^{1+2^n}(b)
%
\begin{align*}
&t(\lambda)=\lambda\nu^{2^n},\,\,\,t(\nu) =\lambda^{2^n}\nu^{2^{n}+1},\\
&t(u)=F(u,v^{2^{ns}})=u+v^{2^{ns}}+(uv^{2^{ns}})^{2^{s-1}},\\
&t(v)=F(v,F(u^{2^{ns}},v^{2^{ns}})).\\
\end{align*}

\medskip

$ G_7 $: Let $\bar{\lambda}$ be the complex conjugate to $\lambda$ and   
$$
\bar{u}=[-1]_F(u)=e(\bar{\lambda}),\,\,\, \bar{v}=[-1]_F(v)=e(\bar{\nu}).
$$

The involution acts as follows 
%t(\lambda)(a)=\lambda(cac)=\lambda(a^{-1}b^{2^n})=\bar{\lambda}(a)
%t(\lambda)(b)=]lambda(cbc)=\lambda(b^{-1})=1=\bar{\lambda}(b)
%t(\nu)(a)=\nu(a^{-1}b^{2^n})=-1=\lambda^{2^n}\bar{\nu}(a)
%t(\nu)(b)=\nu(b^{-1})=\lambda^{2^n}\bar{\nu}(b)
%

\begin{align*}
%\label{G7--2}
&t(\lambda) =\bar{\lambda},\\
%\label{G7--1}
&t(\nu) =\lambda^{2^n}\bar{\nu},\\
%\label{G7--3}
&t(u)=\bar{u}\equiv u+(u\bar{u})^{2^{s-1}} mod(1+t), &&\text{ by }\,\, \eqref{FGL} \text{ as } F(u,\bar{u})=0\\
%\label{G7--4}
&t(v)=F(\bar{v},u^{2^{ns}})=\bar{v}+u^{2^{ns}}+(\bar{v}u^{2^{ns}})^{2^{s-1}}, &&\text{ by \eqref{FGL2^s} }.
\end{align*}

It follows 
\begin{equation*}
\label{F(u,tu)}
0=u+\bar{u}\,\,\, mod (u\bar{u})^{2^{s-1}}\equiv u+\bar{u} \,\,\,mod (u^{2^s})
\end{equation*}

therefore
\begin{equation*}
\label{G7-x_1}
\bar{x}_1^{2^{ns}}=(u+\bar{u})^{2^{ns}}=0, \text{ as } u^{2^{(n+1)s}}=0.\\
\end{equation*}

Then as $u\bar{u}=\bar{x}_2$ is nilpotent we can eliminate $\bar{x}_2^{2^i}=(u\bar{u})^{2^{i}}$ for $i>s-1$ in \eqref{FGL} after finite steps of iteration and write $\bar{x}_2^{2^{s-1}}$ as a polynomial in $u+\bar{u}=\bar{x}_1$. We will not need this polynomial explicitly but only 

\begin{equation*}
\bar{x}_2^{2^{s-1}}\equiv 0 \mod (1+t).
\end{equation*}

For $\bar{y}_1^{2^{ns}}=0$ apply the formula for $t(v)$ and take into account $v+\bar{v}\equiv 0 \mod v^{2^{s}}$.

For the decomposition of  $\bar{y}_2^{2^{s-1}}$ note we have two formulas for $F(v,t(v))=e(\lambda^{2^n})=u^{2^{ns}}$, one is \eqref{v^2^m} and another is \eqref{FGL}. Equating these formulas we have an expression of the form  
\begin{equation*}
\bar{y}_2^{2^{s-1}}=u\bar{x_1}^{2^{ns}-1}+P(\bar{y}_1,\bar{y}_2), \text{ for some polynomial } P(\bar{y}_1,\bar{y}_2).
\end{equation*}

Again as $\bar{y}_2$ is nilpotent we can eliminate $\bar{y}_2^{2^i}$ for $i>s-1$ in \eqref{FGL} after finite steps of iteration and write $\bar{y}_2^{2^{s-1}}$ in the suggested basis. 
Again we only will  
need that
\begin{equation*}
\label{y_2-G7}
\bar{y}_2^{2^{s-1}}\equiv u\bar{x_1}^{2^{ns}-1} \mod Im(1+t).
\end{equation*}

This completes the proof for $G_7$. The proofs for $G_8$ and $G_{11}$ is analogous. Let us sketch the necessary information for the interested reader to produce detailed proofs.

$G_8$: the action of the involution is as follows
%t(\lambda)(a)=\lambda(cac)=\lambda(a^{-1+2^n}b^{2^n})=\lambda^{-1+2^n}(a)
%t(\lambda)(b)=\lambda(cbc)=\lambda(b^{-1+2^n})=1=\lambda^{-1+2^n}(b)
%t(\nu)(a)=\nu(cac)=\nu(a^{-1+2^n}b^{2^n})=\lambda^{2^n}\nu^{-1+2^n}(a)
%t(\nu)(b)=\nu(cbc)=\nu(b^{-1+2^n})=\lambda^{2^n}\nu^{-1+2^n}(b)
%
\begin{align*}
&t(\lambda) =\bar{\lambda}\lambda^{2^n},\,\,\,
t(\nu)=\bar{\nu}\lambda^{2^n}\nu^{2^n},\,\,\,\\
&t(u)=F(\bar{u},u^{2^{ns}}),\\
&t(v)=F(\bar{v},F(u^{2^{ns}},v^{2^{ns}})).   \\
\end{align*}

$G_{11}$: one has 
%t(\lambda)(a)=\lambda(cac)=\lambda(a^{-1}b^{2^n})=\bar{\lambda}\nu^{2^n}(a)
%t(\lambda)(b)=\lambda(cbc)=\lambda(a^{2^n}b^{-1+2^n})
%=-1=\bar{\lambda}\nu^{2^n}(b)
%t(\nu)(a)=\nu(cac)=\nu(a^{-1}b^{2^n})=-1=\lambda^{2^n}\nu^{-1+2^n})(a)
%t(\nu)(b)=\nu(cbc)=\nu(a^{2^n}b^{-1+2^n})=\lambda^{2^n}\nu^{-1+2^n})(b)
\begin{align*}
&t(\lambda) =\bar{\lambda}\nu^{2^n}, \,\,\,
t(\nu)=\bar{\nu}\lambda^{2^n}\nu^{2^n},\,\,\,\\
&t(u)=F(\bar{u},v^{2^{ns}}),   \\
&t(v)=F(\bar{v},F(u^{2^{ns}},v^{2^{ns}})).  \\
\end{align*}

For both cases to get $\bar{x}_1^{2^{ns}}=0$ apply formula for $t(u)$ and $u+\bar{u}\equiv 0 \mod u^{2^{s}}$. Similarly for $\bar{y}_1^{2^{ns}}=0$.
For the decompositions of $\bar{x}_2^{2^{s-1}}$ and  $\bar{y}_2^{2^{s-1}}$   apply \eqref{FGL} and \eqref{v^2^m}.  In particular for $G_8$ we have by \eqref{FGL}  $\bar{x}_2^{2^{s-1}}\equiv u^{2^{ns}}$ modulo some $\bar{x}_1f(\bar{y}_1,\bar{x}_2)\in Im(1+t)$. Therefore $\bar{x}_2^{2^{ns-1}}\equiv 0 \mod(1+t)$ and
by \eqref{v^2^m} for $u$, we have

$$\bar{x}_2^{2^{s-1}}\equiv u^{2^{ns}}\equiv\bar{x}_1^{2^{ns}-1}u+\bar{x}_2^{2^{ns-1}}\equiv \bar{x}_1^{2^{ns}-1}u \,\,\mod(1+t).$$

Similarly  $\bar{y}_2^{2^{ns-1}}\equiv 0 \mod(1+t)$ and we get

$$\bar{x}_2^{2^{s-1}}\equiv F(u^{2^{ns}},v^{2^{ns}})\equiv\bar{x}_1^{2^{ns}-1}u+\bar{y}_1^{2^{ns-1}}v\,\,\ mod(1+t).
$$

Thus we obtain

\begin{align*}
%\label{7,8,11}
&\bar{x}_1^{2^{ns}}=\bar{y}_1^{2^{ns}}=0, &\text { if } G=G_7, G_8, G_9, \\
&\bar{x}_2^{2^{s-1}}\equiv 0,\,\,\,\,\,\,
\bar{y}_2^{2^{s-1}}\equiv \bar{x}_1^{2^{ns}-1}u \mod(1+t), &\text { if } G=G_7, \\
&\bar{x}_2^{2^{s-1}}\equiv \bar{x}_1^{2^{ns}-1}u,\,\,\,\,\,\, \bar{y}_2^{2^{s-1}}\equiv \bar{x}_1^{2^{ns}-1}u+\bar{y}_1^{2^{ns}-1}v
 \mod(1+t), &\text { if }G=G_8,  \\
& \bar{x}_2^{2^{s-1}}\equiv \bar{y}_1^{2^{ns}-1}v,\,\,\,\,\,\,
\bar{y}_2^{2^{s-1}}\equiv \bar{x}_1^{2^{ns}-1}u+\bar{y}_1^{2^{ns}-1}v  \mod(1+t), &\text { if } G=G_{11}.
\end{align*}

\end{proof}

 \medskip

 \begin{lemma}
 \label{invariants}
 Let
 $g=f_0+f_1u+f_2v+f_3uv\in K(s)^*(BH),\,\,\,where f_i=f_i(\bar{x_1},\bar{y_1}\bar{x_2},\bar{y_2})$ are some polynomials written
 uniquely in the monomials $x^{\omega}$ of Lemma \ref{specialbase}.  Then $g$ is invariant under involution $t\in G/H$ iff 
 \begin{equation*}
 \label{conditions}
 f_3\bar{x_1}=f_3\bar{y_1}=0;\,\,f_1\bar{x_1}=f_2\bar{y_1}.
 \end{equation*}
 \end{lemma}
 
 \begin{proof} We have $g$ is invariant iff $g \in Ker(1+t)$. Then
 \begin{align*}
 g+t(g)=&f_1(u+t(u))+f_2(v+t(v))+f_3(uv+t(uv))=\\
 &f_1\bar{x_1}+f_2\bar{y_1}+f_3(\bar{x_1}\bar{y_1}+\bar{x_1}v+\bar{y_1}u)\\
 \end{align*}
 and using Lemma \ref{specialbase} the result follows.
 
 \end{proof}

\medskip

To prove Theorem \ref{main} it suffices to see that all invariants are represented by good elements. 
It is obvious for the elements $a+t(a)=\pi^*Tr^*(a)$ in free summand $[F]^{C_2}$ in \eqref{free-trivial}. Therefore one can work modulo $Im(1+t)$ and check the elements in trivial summand $T$. Let us finish the proof of Theorem \ref{main} by  Propositions \ref{T}, i). We will turn to Proposition \ref{T} ii) later.

\begin{proposition}
\label{T}
Let $T'$ be spanned by the set 
 
\begin{align*}
\text{ for } & G_3,&\\
&\{\bar{x}_2^j\bar{y}_2^l, \,\,\,\bar{x}_2^j\bar{y}_2^lu,\,\,\, \bar{y}_1^{2^s-1}\bar{x}_2^j\bar{y}_2^lv,\,\,\,   
\bar{y}_1^{2^s-1} \bar{x}_2^j\bar{y}_2^luv\, |\, j < 2^{ns-1},\,\,\, l < 2^{(n+1)s-1} \},\\ 
\text{for } &G_4, G_9,& \\
&\{\bar{x}_2^i\bar{y}_2^j, \,\,\, 
\bar{x}_1^{2^{s}-1}\bar{x}_2^i\bar{y}_2^ju,\,\,\,
\bar{y}_1^{2^{s}-1}\bar{x}_2^i\bar{y}_2^jv,\,\,\,
\bar{x}_1^{2^{s}-1}\bar{y}_1^{2^{s}-1}\bar{x}_2^i\bar{y}_2^juv \,\,|\,\, i,j <2^{ns-1}\},\\
\text{ for  } &G_7, G_8, G_{11},&\\
&\{\bar{x}_2^i\bar{y}_2^j,\, 
\bar{x}_1^{2^{ns}-1}\bar{x}_2^i\bar{y}_2^ju, \,\bar{y}_1^{2^{ns}-1} \bar{x}_2^i\bar{y}_2^jv,\, \bar{x}_1^{2^{ns}-1}\bar{y}_1^{2^{ns}-1}\bar{x}_2^i\bar{y}^juv |\,\,\, i,j < 2^{s-1} \}. \\
\end{align*}
Then 

i) All terms in $T'$ are represented by good elements and $T\subset T'$.

ii) $T=T'$.

\end{proposition}

%$G_3$:
%$$
%\{x^{\omega}\}=
%\{\bar{x}_2^j\bar{y}_1^k\bar{y}_2^{l}| j<2^{ns-1}, k<2^s,l<2^{(n+1)s-1}\} 
%$$

Proof of i).
	
$G_3$. The basis set of $T'$ above is suggested by Lemma \ref{specialbase} and Lemma \ref{invariants}: it is clear that all its terms are invariants. The terms $\bar{x}_2^j\bar{y}_1^k\bar{y}_2^{l}\in Im(1+t)$, $k>0$ are omitted as we work modulo $1+t$.  Then all the restrictions follow
by
\begin{equation*}
\bar{y}_1^{2^s}=0,\,\,\,\bar{x}_1=0,\,\,\,\bar{y}_2^{2^{(n+1)s-1}}=0,\,\,\,
\bar{x}_2^{2^{ns-1}}\equiv v\bar{y_1}^{2^s-1} \mod(1+t). 
\end{equation*} 

So that $T\subset T'$. Let us check that  $T'$ is generated  by the images of Euler classes under $\pi^*$, where $\pi$ is the double covering \eqref{pi}.

By definitions 
%$$\pi^*(\alpha)=\lambda^{2^n}\,\,\,\pi^*(a)=\bar{a}=u^{2^{ns}},$$
\begin{align*}
&\pi^*(\alpha)=\lambda^{2^n},  &&\,\,\pi^*(det\pi_!(\nu)\otimes \alpha)=\nu \lambda^{2^n} \nu  \lambda^{2^n}=\nu^2,\\
&\pi^*(v')=v^{2^s}, \,\,\,&&\text{ where}\,\,\,v'=e(det\pi_!(\nu)\otimes\alpha). 
\end{align*}

Taking into account \eqref{v^2^m}, for $m=s$, we get  
\begin{equation}
\label{v3}
\pi^*(v')=v^{2^s}=v\bar{y_1}^{2^{s}-1}+
\sum_{i=1}^{s}\bar{y_1}^{2^{s}-2^i}\bar{y_2}^{2^{i-1}}=
\bar{y}_2^{2^{s-1}}+v\bar{y_1}^{2^{s}-1}\,\,\, \mod(1+t).
\end{equation}

By definition $\bar{x}_2=\pi^*(x_2)$ and $\bar{y}_2=\pi^*(y_2)$. Combining with \eqref{v3} this implies that all elements of the first and third parts of the basis set of $T'$ are $\pi^*$ images of the sums of Euler classes.

For the rest parts of the basis of $T'$ note, that the bundle $\lambda$ can be extended to a bundle over $BG$, say $\lambda'$, represented by 
$\lambda'(\mathbf{a})=e^{2\pi i/2^{n+1}}$, $\lambda'(\mathbf{b})=\lambda'(\mathbf{c})=1$. 
So $\pi^*(e(\lambda'))=u$. Then note that the second and last parts is obtained by multiplying by $u$ of the first and third parts respectively. So that we can easily read off all elements as $\pi^*$ images of the sums of Euler classes.

\medskip

$G_4$. Again the basis for $T'$ is suggested by by Lemma \ref{specialbase}: we have $\bar{x}_1^{2^s}=\bar{y}_1^{2^s}=0$ 
and $\bar{x}_2^{2^{ns-1}}$ and $\bar{y}_2^{2^{ns-1}}$ are decomposable. 
Then applying \eqref{v^2^m} we get

\begin{align*}
&\pi^*(det(\pi_!\nu)\otimes \alpha)=\nu^2,\,\,\,\pi^*(e(det(\pi_!\nu)\otimes \alpha))=v^{2^s}\equiv v\bar{y}_1^{2^s-1}+\bar{y}_2^{2^{s-1}} \mod(1+t),\\
&\pi^*(det(\pi_!\lambda)\otimes \alpha\beta)=\lambda^2,\,\,\, \pi^*(e(det(\pi_!\lambda)\otimes \alpha\beta))=u^{2^s}\equiv u\bar{x}_1^{2^s-1}+\bar{x}_2^{2^{s-1}} \mod(1+t).
\end{align*} 

Thus $G_4$ is good. The proof for $G_9$ is completely analogous.

\medskip

$G_7,G_8,G_{11}$:  It is clear that all elements of the basis elements for $T'$ are invariants and all restrictions are explained by Lemma \ref{specialbase}. It suffices to check that all elements are represented by images of the sums of Euler classes.

$G_7$. The bundle $\lambda^{2^n}$ and $\nu^{2^n}$ can be extended to the line bundles over $BG$, say $\lambda'$ and $\nu'$ respectively. So that 
$$
\pi^*(e(\nu'))=e(\nu^{2^n})=v^{2^{ns}} \text{ and } \,\,\,\pi^*(e(\lambda'))=e(\lambda^{2^n})=u^{2^{ns}}.
$$
Applying again \eqref{v^2^m} we get 

\begin{align*}
\pi^*e(\lambda')=&u^{2^{ns}}=u\bar{x_1}^{2^{ns}-1}+
\sum_{i=1}^{ns}\bar{x_1}^{2^{ns}-2^i}\bar{x_2}^{2^{i-1}}\equiv \\ &u\bar{x_1}^{2^{ns}-1}+\bar{x}_2^{2^{ns-1}} \mod(1+t) \equiv \\
&u\bar{x_1}^{2^{ns}-1} \mod(1+t) \text { by Lemma }\,\,\, \ref{specialbase}\\
\end{align*}

Similarly, applying Lemma \ref{specialbase} we have for $G_8$  

\begin{align*}
&\pi^*(e(det(\pi_!\lambda))=u^{2^{ns}}\equiv \bar{x}_1^{2^{ns}-1}u \,\,\mod(1+t),\\
&\pi^*(e(det(\pi_!\nu))=F(u^{2^{ns}},v^{2^{ns}})\equiv\bar{x}_1^{2^{ns}-1}u+\bar{y}_1^{2^{ns-1}}v\,\,\ mod(1+t)
\end{align*}

and for $G_{11}$

\begin{align*}
&\pi^*(e(det(\pi_!\lambda))=v^{2^{ns}}\equiv\bar{y}_1^{2^{ns}-1}v\,\,\, mod(1+t),\\
&\pi^*(e(det(\pi_!\nu))=F(u^{2^{ns}},v^{2^{ns}})\equiv\bar{x}_1^{2^{ns}-1}u +\bar{y}_1^{2^{ns}-1}v \,\,\,mod(1+t).\\
\end{align*}

This completes the proof of Theorem \ref{main}.

\qed

\bigskip

Proposition \ref{T} ii) may have an independent interest. Let us sketch the proof.

Using the Euler characteristic formula of \cite{HKR}, Theorem D,
one can compute $K(s)^*$-Euler characteristic
$$\chi_{2,s}(G)=rank_{K(s)^*}K(s)^{even}(BG),$$ 
for the classifying spaces of the groups in the title. The answer is as follows. 

\begin{align*}
&\text{group}      &&\chi_{2,s} \\
&G_1               &&2^{(2n+3)s},\\	&G_2, G_4, G_9     &&2^{2(n+1)s-1}-2^{2ns-1}+2^{(2n+1)s},\\	&G_3,G_{10}          &&3 \cdot 2^{2(n+1)s-1}-2^{(2n+1)s-1},\\	&G_5, G_6, G_7, G_8, G_{11}, G_{12} &&2^{2(n+1)s-1}-2^{2s-1}+2^{3s}, \\
&G_{13}, G_{16}                     &&2^{2(n+1)s-1}-2^{(n+2)s-1}+2^{(n+3)s}, \\	
&G_{14}, G_{15}, G_{17}             &&2^{2(n+1)s-1}-2^{(n+1)s-1}+2^{(n+2)s}.
\end{align*}

As $T \subset T'$ it suffices to prove $\chi_{2,s}(T)=\chi_{2,s}(T')$.  It is easily checked the relation between the size of the trivial summand $x=\chi_{2,s}(T)$ and $\chi_{2,s}(G)$ for all groups under consideration
\begin{equation}
\label{sizeT}
(\chi_{2,s}(H)-x):2+2^sx=\chi_{2,s}(G). 
\end{equation}
 Therefore it suffices to see that the number of basis elements of $T'$ in Lemma \ref{T} i)  

\begin{align*}
&\text{T'}      &&\chi_{2,s} (T')\\
&G_3               &&2^{(2n+1)s},\\	
&G_4, G_9     &&4^{ns},\\
&G_7, G_8, G_{11}  &&4^{s}.
\end{align*}

is equal to $x$ in \eqref{sizeT} for all cases.

\qed

%
% $G_3$ 
%\begin{equation*}
%(4^{(n+1)s}-2^{(2n+1)s}):2+2^s\cdot  2^{(2n+1)s}=
%2^{2(n+1)s-1}-2^{(2ns+1)s}+2^{2(n+1)s}=
%3\cdot 2^{2(n+1)s-1}-2^{(2n+1)s-1}=\chi(G_3).
%\end{equation*}

%
%$G_4,G_9$
%$$
%\chi(G)=(4^{(n+1)s}-4^{ns}):2+2^s\cdot 4^{ns}
%=2^{2(n+1)s-1}-2^{2ns-1}+2^{2ns+s}.
%$$

%
%$G_7,G_8,G_{11}$
%$$\chi(G)=(4^{(n+1)s}-2^{2s}):2+2^s\cdot %2^{2s}=2^{2(n+1)s-1}-2^{2s-1}+2^{3s}
%$$
%

%\section*{Acknowledgements}
%The authors are very grateful to the referee for exceptionally thorough %analysis of the paper and numerous important suggestions which have been %very useful for improving the paper.

\bibliographystyle{amsplain}

\end{document}